\documentclass[11pt]{amsart}
\usepackage{amssymb,color}
\usepackage{amsmath}
\usepackage{latexsym}
\usepackage{verbatim}
\usepackage[left=1.3in,right=1.3in,top=1in,bottom=1in]{geometry}
\usepackage[linktocpage=true]{hyperref}

\hypersetup{ colorlinks   = true, %Colours links instead of ugly boxes
urlcolor  = blue, %Colour for external hyperlinks
linkcolor    = black, %Colour of internal links
citecolor   = black %Colour of citations
}

\newcommand{\Z}{{\mathbb Z}}

\newcommand{\D}{{\mathbb D}}
\newcommand{\T}{{\mathbb T}}

\newcommand{\pD}{{\partial\mathbb{D}}}

\newcommand{\CC}{{\mathcal C}}

\newcommand{\CE}{{\mathcal E}}

\newcommand{\CI}{{\mathcal I}}

\newcommand{\CZ}{{\mathcal Z}}

\DeclareMathOperator{\dist}{{dist}}

\newtheorem{theorem}{Theorem}[section]

\theoremstyle{definition}
\newtheorem*{remark}{Remark}
\theoremstyle{definition}
\newtheorem{defi}[theorem]{Definition}

\newtheorem{Theorem}{Theorem}
\newtheorem{prop}[theorem]{Proposition}

\newtheorem{lemma}[theorem]{Lemma}
\numberwithin{equation}{section}

\begin{document}

%%% Fengpeng
\author[F.\ Wang]{Fengpeng Wang}
\address{Ocean University of China, Qingdao 266100, Shandong, China and Rice University, Houston, TX~77005, USA}
\email{wfpouc@gmail.com}
\thanks{F.W.\ was supported by CSC (No.201606330003) and NSFC (No.11571327).}

\title{A formula related to CMV matrices and Szeg\H o cocycles}

\begin{abstract}
For Schr\"odinger operators, there is a well known and widely used formula connecting the transfer matrices and Dirichlet determinants. No analog of this formula was previously known for CMV matrices. In this paper we fill this gap and provide the CMV analog of this formula.
\end{abstract}

\maketitle

\section{Introduction}

In recent years, orthogonal polynomials on the unit circle (OPUC) have been extensively studied, see \cite{onefoot} for a expository note and books \cite{OPUC1, OPUC2} for details. In the study of orthogonal polynomials on the real line (OPRL), Jacobi matrix representation is one of the key tools, hence people want to get the matrix realization of OPUC. While for OPUC case, different orthonomal bases corresponds to different matrix representation. In 2003, M. Cantero, L. Moral, L. Vel\'azquez \cite{cmv05} gave the "right" basis and the corresponding matrix representation which is named after them, i.e. CMV matrix. Naturally, it is viewed as the unitary analog of Jacobi matrices. 

Since Jacobi matrices have been studied for more than one hundred years, there are fruitful results in this area and their relation with OPRL is clear. As CMV matrices are the unitary analog of Jacobi matrices, people expect the results for Jacobi matrices also hold for CMV matrices. Indeed, many of them have been carried out in Barry Simon's monographs \cite{OPUC1, OPUC2}, while the formula in our paper is one of the exceptions.

For a special class of Jacobi operators, one dimensional Schr\"odinger operators, they can be viewed as tridiagonal matrices and people related its Dirichlet determinants to the associated transfer matrices with an eqaution, which is playing an important role in the study of Schr\"odinger operators, for example it can be used to prove Anderson localization. The formula in our paper is the CMV analog of this relation, which means it might be useful to get Anderson localization of CMV matrices and indeed it is.

The importance of Anderson localization is due to the seminal work by the physicist P. W. Anderson \cite{anderson}, which is named after him and helped him get the 1977 physics Nobel prize. In physics, Anderson localization refers to the phenomena that disorder in the media will cause suppression of electron transport; while in mathematics, it means the corresponding operator has only pure point spectrum with exponentially decaying eigenfunctions.

Mathematically rigorous studies of the Anderson Model and other models started in the 1970s and several powerful methods have been found to prove Anderson localization, such as {\it multiscale analysis} (MSA) introduced by J. Fr\"ohlich and T. Spencer \cite{FS83} , {\it fractional moments method} (FMM) developed by M. Aizenman and S. Molchanov \cite{AM93}, etc. Recently, the method developed by J. Bourgain, M. Goldstein and W. Schlag \cite{BG00, BS00} for one-dimensional Schr\"odinger operators has been applied widely to other one-dimensional models \cite{7authors17, DW17, K05, K14, K17}, which motivated us to apply this method to CMV matrices.

\section{Preliminaries}
In this section, we recall the Schr\"odinger version of this formula and give some preparations for CMV matrix and Szeg\H o cocycle map. 
\subsection{Schr\"odinger case}
Consider the lattice Schr\"odinger operator $H_{v}$ acting on $\ell^{2}(\mathbb{Z})$
\begin{equation*}
[H_{v}u](n)=u(n+1)+u(n-1)+v_{n}u(n).
\end{equation*}
Then the solutions to $u(n+1)+u(n-1)+v_{n}u(n)=Eu_{n}$ must satisfy
\begin{equation*}
\begin{pmatrix} u(n+1) \\ n(n)  \end{pmatrix} = M_{v_{n}}^{E} \begin{pmatrix} n(n) \\ u(n-1) \end{pmatrix}
\end{equation*}
where 
\begin{equation*}
M_{v_{n}}^{E} = \begin{pmatrix} E-v_{n} & -1 \\ 1 & 0 \end{pmatrix}
\end{equation*}
is usually referred to as Schr\"odinger cocycle map.

It is well known that a Schr\"odinger operator can be viewed as a tridiagonal bi-infinite matrix. Let $P_{[a,b]}$ denote the projection $\ell^{2}(\Z) \rightarrow \ell^{2}([a,b])$ and define the restriction of Schr\"odinger operator by
\begin{equation}{\label{restriction}}
H_{v,[a,b]} = (P_{[a,b]})^{*} H_{v} P_{[a,b]}
\end{equation}
where $a < b$ and $a,b \in \Z$.
Then the restriction of Schr\"odinger operator on interval $[1,n]$ is equivalent to 
\begin{equation*}
    H_{v,[1,n]}=
    \begin{pmatrix}
        v_{1} & 1 & 0 & \cdots & 0\\
        1 & v_{2} & 1 & \cdots & 0\\
        0 & 1 & v_{3} & \cdots & 0\\
        \vdots & \vdots & \vdots & \ddots & \vdots\\
        0 & 0 & 0 & \cdots & v_{n}
    \end{pmatrix}
\end{equation*}
where $\{v_{j}\}_{j \in \mathbb{Z}}$ are the potentials and the n-step transfer matrix is given by
\begin{equation*}
    M_{n}^{E} = \prod_{j=n}^{1}
                    \begin{pmatrix}
                        E - v_{j}  & -1 \\
                           1      & 0  
                    \end{pmatrix}.
\end{equation*}

By induction, it's not hard to get the following formula which is well known

\begin{equation}{\label{sformula}}
    M_{n}^{E} = 
      \begin{pmatrix}
         \det(E-H_{v,[1,n]})  & -\det(E-H_{v,[2,n]}) \\
            \det(E-H_{v,[1,n-1]})     &  -\det(E-H_{v,[2,n-1]})  
      \end{pmatrix}.
\end{equation}

Next, before giving the CMV analog of this formula, we need some general settings about CMV matrices.

\subsection{Definitions and notations related to CMV matrices}
In this section, we will introduce CMV matrix in view of orthogonal polynomial on the unit circle. 

Let the probability measure $\mu$ on the unit circle $\partial\mathbb{D}$ be \emph{nontrivial}, which means its support contains infinitely many points and denote the monic orthogonal polynomials by $\Phi_{n}(z)$.

\begin{defi}
For any polynomial $Q_{n}(z)$ of degree $n$, define the \emph{reversed polynomial}  $Q^{*}_{n}(z)$ (also called \emph{Szeg\H o dual} \cite{cmvfive})  by the following equation,
\begin{equation*}
    Q^{*}_{n}(z)=z^{n} \overline{Q_{n}(1/\bar{z})},
\end{equation*}
Specially, for $z \in \partial\mathbb{D}$,
$$Q_{n}^{*}(z)=z^{n} \overline{Q_{n}(z)}.$$

\end{defi}

Then the \emph{Szeg\H o recurrence} is given by
\begin{equation*}
    \Phi_{n+1}(z) = z \Phi_{n}(z) - \bar{\alpha}_{n} \Phi^{*}_{n}(z)
\end{equation*}
where the parameters $\alpha_{0}, \alpha_{1}, \cdots$ are called \emph{Verblunsky coefficients} and they are all in the unit disc $\mathbb{D} = \{ z \in \mathbb{C}: \lvert z \rvert < 1 \}$.

The \emph{half-line CMV matrix} associated with Verblunsky coefficients $\{\alpha_{n}\}_{n \in \mathbb{N}}$ is given by
\begin{equation*}
    \mathcal{C}=
    \begin{pmatrix}
         \bar{\alpha}_{0} &  \bar{\alpha}_{1} \rho_{0}     &  \rho_{1} \rho_{0}               &                 &             &             &   \\
         \rho_{0}         &  -\bar{\alpha}_{1} \alpha_{0}  & -\rho_{1} \alpha_{0}             &                 &                &      &   \\
                          &  \bar{\alpha}_{2} \rho_{1}     & -\bar{\alpha}_{2} \alpha_{1}     &  \bar{\alpha}_{3} \rho_{2} & \rho_{3} \rho_{2} &        &   \\
                          &  \rho_{2} \rho_{1}             & -\rho_{2} \alpha_{1}             & -\bar{\alpha}_{3} \alpha_{2}  &  -\rho_{3} \alpha_{2}&     &  \\
                   &       &        &  \bar{\alpha}_{4} \rho_{3}  &  -\bar{\alpha}_{4} \alpha_{3}  & \bar{\alpha}_{5}\rho_{4} &     \\
                   &       &        &  \rho_{4}\rho_{3} & -\rho_{4}\alpha_{3}  & -\bar{\alpha}_{5}\alpha_{4}   &     \\
                     &    &   &    & \ddots  & \ddots  & \ddots  
    \end{pmatrix}
\end{equation*}
where $\rho_{n} = (1-|\alpha_{n}|^{2} )^{1/2}$, so it defines a unitary operator in $\ell^{2}(\mathbb{N})$. 

According to Verblunsky's theorem (also called Favard's theorem for the circle), the map $\mu \rightarrow \{\alpha_{n}\}_{n \in \mathbb{N}}$ sets up a one-one correspondence between the set of nontrivial probability measures on $\pD$ and $\times_{j=0}^{\infty}\mathbb{D}$.

Similarly, an \emph{extended CMV matrix} is a unitary operator on $\ell^{2}(\mathbb{Z})$ defined by a bi-infinite sequence $\{\alpha_{n}\}_{n \in \mathbb{Z}} \subset \mathbb{D}$,
\begin{equation*}
    \mathcal{E}=
    \begin{pmatrix}
         \ddots & \ddots & \ddots &     &   &    &   &     \\
                & -\bar{\alpha}_{0}\alpha_{-1}  & \bar{\alpha}_{1}\rho_{0}  & \rho_{1}\rho_{0} &  &   &   &    \\
                & -\rho_{0}\alpha_{-1}  & -\bar{\alpha}_{1}\alpha_{0}  &  -\rho_{1}\alpha_{0} &  &   &    &      \\
                &   & \bar{\alpha}_{2}\rho_{1} & -\bar{\alpha}_{2}\alpha_{1}  & \bar{\alpha}_{3}\rho_{2}  & \rho_{3}\rho_{2} &       &           \\
                &   & \rho_{2}\rho_{1} & -\rho_{2}\alpha_{1}  & -\bar{\alpha}_{3}\alpha_{2}  &  -\rho_{3}\alpha_{2} &       &               \\
                &   &   &    & \bar{\alpha}_{4}\rho_{3} & -\bar{\alpha}_{4}\alpha_{3}  & \bar{\alpha}_{5}\rho_{4}  &         \\
                &   &   &    & \rho_{4}\rho_{3} & -\rho_{4}\alpha_{3}  & -\bar{\alpha}_{5}\alpha_{4}  &                \\
                &   &   &    &    &  \ddots & \ddots & \ddots
    \end{pmatrix}
\end{equation*}

In the expressions of $\CC$ and $\CE$, all unspecified matrix entries are implicitly assumed to be zero.

\subsection{Szeg\H o cocycle map}

Recall Szeg\H o recurrence, it's equivalent to the following one,
\begin{equation}{\label{szegorecursion}}
    \rho_{n} \varphi_{n+1}(z) = z \varphi_{n}(z) - \bar{\alpha}_{n}\varphi_{n}^{*}(z)
\end{equation}
where
\begin{equation*}
    \varphi_{n}(z) = \frac{\Phi_{n}(z)}{\lVert \Phi_{n}(z) \rVert}.
\end{equation*}

Apply $^{*}$ to both sides of (\ref{szegorecursion}), then
\begin{equation}{\label{szegorecursion2}}
    \rho_{n} \varphi_{n+1}^{*}(z) =  \varphi_{n}^{*}(z) - \alpha_{n} z\varphi_{n}(z).
\end{equation}
Equations (\ref{szegorecursion}) and (\ref{szegorecursion2}) can be written as
\begin{equation*}
    \begin{pmatrix}
        \varphi_{n+1} \\
        \varphi_{n+1}^{*}
    \end{pmatrix}= S^{z}_{\alpha_{n}}
    \begin{pmatrix}
        \varphi_{n} \\
        \varphi_{n}^{*}
    \end{pmatrix}
\end{equation*}
where 
\begin{equation*}
    S^{z}_{\alpha_{n}}=\frac{1}{\rho_{n}}\begin{pmatrix}
              z     &   -\bar{\alpha}_{n} \\
              -\alpha_{n} z & 1
    \end{pmatrix}
\end{equation*}
is the {\it Szeg\H o cocycle map} and the $n$-step transfer matrix is defined by
\begin{equation}{\label{nstep}}
    S^{z}_{n} = \prod \limits_{j=n-1}^{0} S^{z}_{\alpha_{n}}.
\end{equation}

For $X \in \{\mathcal{C}, \mathcal{E}\}$, let's define $X_{[a,b]}=(P_{[a,b]})^{*} X P_{[a,b]}$ and $\varphi^{X,z}_{[a,b]} = \det(z-X_{[a,b]})$, then the CMV analogs of formula (\ref{sformula}) are as follows 

\begin{Theorem}[Half-line case]{\label{halflinecase}}
For $z \in \pD$, we have 
    \begin{equation}{\label{formulahalf}}
        S^{z}_{n}=\prod\limits_{j=0}^{n-1}\frac{1}{\rho_{j}}
        \begin{bmatrix}
            z\varphi^{\CC,z}_{[1,n-1]}                                     &      \varphi^{\CC,z}_{[0,n-1]}-z\varphi^{\CC,z}_{[1,n-1]}\\
            z(\varphi^{\CC,z}_{[0,n-1]}-z\varphi^{\CC,z}_{[1,n-1]})^{*}        &     (\varphi^{\CC,z}_{[1,n-1]})^{*}
        \end{bmatrix}
    \end{equation}
\end{Theorem}

\begin{Theorem}[Extended case]{\label{extendedcase}}
For $z \in \pD$, we have 
    \begin{equation}{\label{etd1}}
    S^{z}_{n} = \prod\limits_{j=0}^{n-1}\frac{1}{\rho_{j}}
        \begin{bmatrix}
            z\varphi^{\CE,z}_{[1,n-1]}                &    \frac{z\varphi^{\CE,z}_{[1,n-1]}-\varphi^{\CE,z}_{[0,n-1]}}{\alpha_{-1}}\\
            z(\frac{z\varphi^{\CE,z}_{[1,n-1]}-\varphi^{\CE,z}_{[0,n-1]}}{\alpha_{-1}})^{*}   &    (\varphi^{\CE,z}_{[1,n-1]})^{*}
        \end{bmatrix}
    \end{equation}
\end{Theorem}

\begin{remark}
One may notice the term $\frac{1}{\alpha_{-1}}$ in (\ref{etd1}) which means there might be a problem in the case $\alpha_{-1}=0$. But we should mention that the numerator of $S_{n}^{z}(1,2)$ contains a factor $\alpha_{-1}$ and from our proof one can see that the term $S_{n}^{z}(1,2)$ doesn't depend on $\alpha_{-1}$.

     Actually, Theorem \ref{halflinecase} and Theorem \ref{extendedcase} are equivalent, we will first prove the equivalency of them, and then give the proof of Theorem \ref{halflinecase}, since its proof needs some additional preparations.
\end{remark}

\section{Proof}

\begin{proof}[Proof of the equivalency]
Notice the special forms of CMV matrices $\CC$ and $\CE$, it is not hard to see $\CE_{[a,b]}=\mathcal{C}_{[a,b]}$, for $1 \leq a < b$, hence $\CE_{[1,n-1]}=\CC_{[1,n-1]}$. 

Compare the corresponding entries in the matrices from Theorem \ref{halflinecase} and Theorem \ref{extendedcase}, then the equivalency can be got immediately if the following equation holds
\begin{equation}{\label{edhe}}
    \det (z-\mathcal{C}_{[0,n-1]})-z \det (z-\mathcal{E}_{[1,n-1]})= \frac{z \det (z-\mathcal{E}_{[1,n-1]})-\det (z-\mathcal{E}_{[0,n-1]})}{\alpha_{-1}}.
\end{equation}
Simple calculations tell us that
\begin{equation*}
\det (z-\mathcal{E}_{[0,n-1]})=(z+\bar{\alpha}_{0} \alpha_{-1})\det(z-\mathcal{E}_{[1,n-1]})-\rho_{0} \alpha_{-1} \det \mathcal{P}_{n-1}
\end{equation*}
\begin{equation*}
\det (z-\mathcal{C}_{[0,n-1]})=(z-\bar{\alpha}_{0}) \det (z-\mathcal{E}_{[1,n-1]})+\rho_{0} \det \mathcal{P}_{n-1}
\end{equation*}
where the matrix $\mathcal{P}_{n-1}$ is given by
\begin{equation*}
\mathcal{P}_{n-1}=
\begin{pmatrix}
          -\bar{\alpha}_{1} \rho_{0}     &  -\rho_{1} \rho_{0}               &               &               &   &        \\
          -\bar{\alpha}_{2} \rho_{1}     & z+\bar{\alpha}_{2} \alpha_{1}     &  -\bar{\alpha}_{3} \rho_{2} & -\rho_{3} \rho_{2}  &   &        \\
          -\rho_{2} \rho_{1}             & \rho_{2} \alpha_{1}             & z+\bar{\alpha}_{3} \alpha_{2}  &  \rho_{3} \alpha_{2} &   &  \\
               &       &  -\bar{\alpha}_{4} \rho_{3}  &  z+\bar{\alpha}_{4} \alpha_{3}  &    &   \\
            &   &    &    & \ddots  &  \\
            &     &     &     &    &  z+\bar{\alpha}_{n-1} \alpha_{n-2}
\end{pmatrix}.
\end{equation*}
These two equations imply that
    \begin{equation*}
        \begin{split}
          z\det(z-\mathcal{E}_{[1,n-1]})-\det(z-\mathcal{E}_{[0,n-1]}) & = -\bar{\alpha}_{0} \alpha_{-1} \det(z-\mathcal{E}_{[1,n-1]})+\rho_{0} \alpha_{-1}\det \mathcal{P}_{n-1} \\
          & = \alpha_{-1} (-\bar{\alpha}_{0} \det (z-\mathcal{E}_{[1,n-1]})+\rho_{0}\det \mathcal{P}_{n-1})  \\ 
          & = \alpha_{-1} (\det (z-\mathcal{C}_{[0,n-1]})-z\det (z-\mathcal{E}_{[1,n-1]}))
        \end{split}
    \end{equation*}
which means (\ref{edhe}) is true and hence Theorem \ref{halflinecase} is equivalent to Theorem \ref{extendedcase}.
\end{proof}

Next, let's recall some definitions and notations from \cite{OPUC1} which are necessary to obtain the proof of Theorem \ref{halflinecase}.
\begin{defi}
\rm Let $\{\alpha_{n}\}_{n \in \mathbb{Z}}$ be a set of Verblunsky coefficients and $\lambda \in  \partial \mathbb{D}$. We define $\Phi^{\lambda}_{n}$ by
\begin{equation*}
    \Phi^{\lambda}_{n}(z;d\mu)=\Phi^{\lambda}_{n}(z;d\mu_{\lambda})
\end{equation*}
with $d\mu_{\lambda}$, the \emph{Aleksandrov measures}, defined by
\begin{equation*}
    \alpha_{n}(d\mu_{\lambda})=\lambda \alpha_{n}(d\mu)
\end{equation*}

Specially, for the case $\lambda = -1$,
\begin{equation*}
    \Psi_{n}(z;d\mu)=\Phi^{\lambda = -1}_{n}(z;d\mu)
\end{equation*}
are called the \emph{second kind of polynomials} for $\mu$.
\end{defi}

Let $\mathcal{C}'$ denote the CMV matrix whose Verblunsky coefficients replaced by $\{-\alpha_{n}\}_{n \in \mathbb{N}}$, that is
\begin{equation*}
\mathcal{C}'=
\begin{pmatrix}
         -\bar{\alpha}_{0}  &  -\bar{\alpha}_{1} \rho_{0}     &  \rho_{1} \rho_{0}               &          &       &   &   \\
         \rho_{0} &  -\bar{\alpha}_{1} \alpha_{0}  & \rho_{1} \alpha_{0}             &               &            &   &   \\
                         &  -\bar{\alpha}_{2} \rho_{1}     & -\bar{\alpha}_{2} \alpha_{1}     &  -\bar{\alpha}_{3} \rho_{2} & \rho_{3} \rho_{2} &  &   \\
                         &  \rho_{2} \rho_{1}             & \rho_{2} \alpha_{1}             & -\bar{\alpha}_{3} \alpha_{2}  &  \rho_{3} \alpha_{2} &   &   \\
                  &      &       &  -\bar{\alpha}_{4} \rho_{3}  &  -\bar{\alpha}_{4} \alpha_{3}  & -\bar{\alpha}_{5}\rho_{4} &     \\
                   &      &       &  \rho_{4}\rho_{3} & \rho_{4}\alpha_{3}  & -\bar{\alpha}_{5}\alpha_{4}     &     \\
                   &    &   &    & \ddots  & \ddots   & \ddots
\end{pmatrix}
\end{equation*}

\noindent The corresponding extened CMV matrix is

\begin{equation*}
    \mathcal{E}'=
    \begin{pmatrix}
         \ddots & \ddots & \ddots &     &   &    &   &     \\
                & -\bar{\alpha}_{0}\alpha_{-1}  & -\bar{\alpha}_{1}\rho_{0}  & \rho_{1}\rho_{0} &  &   &   &    \\
                & \rho_{0}\alpha_{-1}  & -\bar{\alpha}_{1}\alpha_{0}  &  \rho_{1}\alpha_{0} &  &   &    &      \\
                &   & -\bar{\alpha}_{2}\rho_{1} & -\bar{\alpha}_{2}\alpha_{1}  & -\bar{\alpha}_{3}\rho_{2}  & \rho_{3}\rho_{2} &       &           \\
                &   & \rho_{2}\rho_{1} & \rho_{2}\alpha_{1}  & -\bar{\alpha}_{3}\alpha_{2}  &  \rho_{3}\alpha_{2} &       &               \\
                &   &   &    & -\bar{\alpha}_{4}\rho_{3} & -\bar{\alpha}_{4}\alpha_{3}  & -\bar{\alpha}_{5}\rho_{4}  &         \\
                &   &   &    & \rho_{4}\rho_{3} & \rho_{4}\alpha_{3}  & -\bar{\alpha}_{5}\alpha_{4}  &                \\
                &   &   &    &    &  \ddots & \ddots & \ddots
    \end{pmatrix}
\end{equation*}

\begin{lemma}{\label{lskp}}
Let $\Psi_{n}(z)$ be the second kind of polynomial and $\mathcal{C}'_{[0,n-1]}$ denote the restriction of $\mathcal{C}'$, then
    \begin{equation*}
        \Psi_{n}(z) = \det(z-\mathcal{C}'_{[0,n-1]}).
    \end{equation*}
\end{lemma}
\begin{proof}
    It is known that (Theorem 5.3 of \cite{onefoot})
    \begin{equation*}
        \Phi_{n}(z) = \det(z-\mathcal{C}_{[0,n-1]})
    \end{equation*}

According to the definition of second kind of polynomial, the Verblunsky coefficients in $\Phi_{n}(z)$ and $\Psi_{n}(z)$ have opposite signs. Notice the way we define matrix $\mathcal{C}'$, this Lemma follows by Theorem 5.3 in \cite{onefoot}.
\end{proof}

Recall the definition of $\mathcal{P}_{n-1}$ in the proof of Theorem \ref{extendedcase}, replace $\{\alpha_{j}\}_{0<j<n}$ by $\{-\alpha_{j}\}_{0<j<n}$, we get
\begin{equation*}
\mathcal{P}'_{n-1}=
\begin{pmatrix}
          \bar{\alpha}_{1} \rho_{0}     &  -\rho_{1} \rho_{0}               &                &                &   &        \\
          \bar{\alpha}_{2} \rho_{1}     & z+\bar{\alpha}_{2} \alpha_{1}     &  \bar{\alpha}_{3} \rho_{2} & -\rho_{3} \rho_{2}  &   &       \\
          -\rho_{2} \rho_{1}             & -\rho_{2} \alpha_{1}             & z+\bar{\alpha}_{3} \alpha_{2}  &  -\rho_{3} \alpha_{2} &   & \\
               &       &  \bar{\alpha}_{4} \rho_{3}  &  z+\bar{\alpha}_{4} \alpha_{3}  &   &  \\
            &   &    &    & \ddots  &  \\
           &    &    &    &   &  z+\bar{\alpha}_{n-1} \alpha_{n-2}
\end{pmatrix}
\end{equation*}

\begin{lemma}{\label{lfsp}}
Let $\mathcal{E}_{[1,n-1]}$, $\mathcal{E}'_{[1,n-1]}$, $\mathcal{P}_{n-1}$ and $\mathcal{P}'_{n-1}$ be as above, then we have
\begin{equation*}
    \det(z-\mathcal{E}_{[1,n-1]})=\det(z-\mathcal{E}'_{[1,n-1]})
\end{equation*}
\begin{equation*}
    \det\mathcal{P}_{n-1} = -\det\mathcal{P}'_{n-1}.
\end{equation*}
\end{lemma}
\begin{proof}
    The first equation can be easily proved by induction and the second one follows by a direct calculation. More precisely,

It is easy to check that $\det(z-\mathcal{E}_{[n-2,n-1]})=\det(z-\mathcal{E}'_{[n-2,n-1]})$. For the inductive step, we assume $\det(z-\mathcal{E}_{[j,n-1]})=\det(z-\mathcal{E'}_{[j,n-1]})$ holds for all $k<j<n-1$, then by a direct calculation,
    \begin{equation*}
        \begin{split}
          & \det(z-\mathcal{E}_{[k,n-1]})\\
          & = (z+\bar{\alpha}_{k} \alpha_{k-1})\det(z-\mathcal{E}_{[k+1,n-1]})+\bar{\alpha}_{k+1}\alpha_{k-1} \rho^{2}_{k}\det(z-\mathcal{E}_{[k+2,n-1]})+ \cdots+\bar{\alpha}_{n-1}\alpha_{k-1}\prod\limits_{j=n-2}^{k}\rho^{2}_{j} \\
          & = (z+\bar{\alpha}_{k} \alpha_{k-1})\det(z-\mathcal{E}'_{[k+1,n-1]})+\bar{\alpha}_{k+1}\alpha_{k-1} \rho^{2}_{k}\det(z-\mathcal{E}'_{[k+2,n-1]})+ \cdots+\bar{\alpha}_{n-1}\alpha_{k-1}\prod\limits_{j=n-2}^{k}\rho^{2}_{j}  \\ 
          & = \det(z-\mathcal{E'}_{[k,n-1]}). 
        \end{split}
    \end{equation*}
In particular, take $k = 1$, we get the first equation
\begin{equation*}
    \det(z-\mathcal{E}_{[1,n-1]})=\det(z-\mathcal{E'}_{[1,n-1]}).
\end{equation*}
For the second equation, since $\det(z-\mathcal{E}_{[j,n-1]})=\det(z-\mathcal{E'}_{[j,n-1]})$ holds for all $0 \leq j < n-1$, then
\begin{equation*}
\det (z-\mathcal{E}_{[0,n-1]})=(z+\bar{\alpha}_{0} \alpha_{-1})\det(z-\mathcal{E}_{[1,n-1]})-\rho_{0} \alpha_{-1} \det \mathcal{P}_{n-1}
\end{equation*}
\begin{equation*}
\det (z-\mathcal{E}'_{[0,n-1]})=(z+\bar{\alpha}_{0} \alpha_{-1})\det(z-\mathcal{E}'_{[1,n-1]})+\rho_{0} \alpha_{-1} \det \mathcal{P}'_{n-1}
\end{equation*}
implies $\det\mathcal{P}_{n-1}(z) = -\det\mathcal{P}'_{n-1}(z).$
\end{proof}

\begin{proof}[Proof of Theorem {\rm \ref{halflinecase}}]
     From \cite[Section 3.2]{OPUC1}, we have
    \begin{equation*}
        S^{z}_{n} = \prod\limits_{j=0}^{n-1}\rho^{-1}_{j}
        \begin{bmatrix}
            zB^{*}_{n-1}(z)    &    A^{*}_{n-1}(z)\\
            zA_{n-1}(z)            &    B_{n-1}(z)
        \end{bmatrix}
    \end{equation*}
where $\Phi_{n}(z)=zB^{*}_{n-1}(z)+A^{*}_{n-1}(z)$ and
$$A_{n-1}(z)=\frac{\Phi^{*}_{n}(z)-\Psi^{*}_{n}(z)}{2z}$$
$$B_{n-1}(z)=\frac{\Phi^{*}_{n}(z)+\Psi^{*}_{n}(z)}{2}.$$

Compare this existing result with our Theorem \ref{halflinecase}, it is sufficient to prove 
\begin{equation*}
    z B^{*}_{n-1}(z) = z \det(z-\mathcal{E}_{[1,n-1]})
\end{equation*}
that is, 
\begin{equation}{\label{efsd}}
    \Phi_{n}(z) + \Psi_{n}(z)=2z\det(z-\mathcal{E}_{[1,n-1]}).
\end{equation}
By Lemma \ref{lskp} and some calculations, we have
\begin{equation*}
    \begin{split}
    \Phi_{n}(z)
           & = \det(z-\mathcal{C}_{[0,n-1]}) \\
           & = (z-\bar{\alpha}_{0})\det(z-\mathcal{E}_{[1,n-1]})+\rho_{0}\det\mathcal{P}_{n-1}(z)
    \end{split}
\end{equation*}
and
\begin{equation*}
    \begin{split}
    \Psi_{n}(z)
           & = \det(z-\mathcal{C}'_{[0,n-1]}) \\
           & = (z+\bar{\alpha}_{0})\det(z-\mathcal{E}'_{[1,n-1]})+\rho_{0}\det\mathcal{P}'_{n-1}(z)
    \end{split}
\end{equation*}
Applying Lemma \ref{lfsp}, it is obvious that (\ref{efsd}) holds and hence Theorem \ref{halflinecase}.

\end{proof}

\section{Application}
In this section, we explain how to get Anderson localization of half-line CMV matrices via the method developed in \cite{BG00} where our formula can play an important role, the extended CMV matrices case can be carried out similarly. 

From now on, we consider a special class of Verblunsky coefficients which are generated by a analytic function $\alpha(x) \in \mathbb{D}$, i.e. $\alpha_{n}(x) = \alpha(x+n\omega)$, where $x, \omega \in \mathbb{T}$.

Recall the Szeg\H o cocycle map which is given by
\begin{equation*}
    S^{z}_{\alpha_{n}}(x)=\frac{1}{\rho_{n}(x)}\begin{pmatrix}
              z     &   -\bar{\alpha}_{n}(x) \\
              -z \alpha_{n}(x)  & 1
    \end{pmatrix}.
\end{equation*}
Since $\det S^{z}_{\alpha_{n}}(x) = z \in \pD$ and the method in \cite{BG00} only cares about the norm of n-step transfer matrix, so it's equivalent to study the following one
\begin{equation*}
    M^{z}_{\alpha_{n}}(x)=\frac{1}{\rho_{n}(x)}\begin{pmatrix}
              \sqrt{z}     &   -\frac{\bar{\alpha}_{n}(x)}{\sqrt{z}} \\
              -\sqrt{z} \alpha_{n}(x)  & \frac{1}{\sqrt{z}}
    \end{pmatrix}
\end{equation*}
and the corresponding n-step transfer matrix
\begin{equation*}
    M^{z}_{n}(x) = \prod \limits_{j=n-1}^{0} M^{z}_{\alpha_{n}}(x).
\end{equation*}
Define
\begin{equation*}
L_{n}(x) = \frac{1}{n} \int_{\mathbb{T}}  \log \| M^{z}_{n}(x) \| dx
\end{equation*}
then the Lyapunov exponent is given by
\begin{equation*}
\text{$L(z)=\inf_{n} L_{n}(z)$ for $n \to \infty$.}
\end{equation*}
\begin{theorem}{\label{cmvandersonlocalization}}
Consider the family $\{\CC_{\omega}\}_{\omega \in \mathbb{T}}$ of half-line CMV matrices with Verblunsky coefficients $\{\alpha_{n}(x)\}_{n \in \mathbb{N}}$ generated by analytic function $\alpha: \T \to \D$, that is, $\alpha_{n}(x)=\alpha(x+n\omega)$  where $x,\omega \in \T$. Assume the Lyapunov exponent $L(z)$ satisfies
\begin{equation*}
\text{$L(z) > \delta_{0} >0$ for all $z \in \mathcal{Z}$ and $\omega \in \mathcal{I}$},
\end{equation*}
where $\CZ \subset \pD$ and $\CI \subset \mathbb{T}$ are compact intervals.

Then for almost every $\omega \in \mathcal{I}$, $\CC_{\omega}(0)$ has pure point spectrum on $\CZ$ with exponentially decaying eigenfunctions (i.e. Anderson localization). 
\end{theorem}
\begin{remark}
In the proof this theorem, we use the method developed in \cite{BG00} and replace the equation (\ref{sformula}) there by our formula.

To make our theorem meaningful, we should mention that the assumption in this theorem is guaranteed by \cite{zhangpositive} where Zhenghe Zhang provided an example whose Lyapunov exponent is uniformly positive. More specifically, the analytic function is $\alpha(x)=\lambda e^{2\pi i h(x)}$ where $\lambda \in (0,1)$ and $h(x) \in C^{\omega}(\mathbb{T},\mathbb{T})$.
\end{remark}

\subsection{A large deviation estimate}
For analytic Schr\"odinger cocycle with Diophantine frequencies, the first large deviation estimate (LDT) is established by J. Bourgain and M. Goldstein \cite{BG00}. From the proof we can see that this LDT holds true for all analytic ${\rm SL}(2, \mathbb{R})$ matrices whose n-step transfer matrices satisfying
\begin{equation}{\label{bounded}}
\|M_{n}(x)\| + \|M_{n}^{-1}(x)\| \leq C^{n}.
\end{equation}
It is well known that $M^{z}_{\alpha_{n}}(x)$ is unitary equivalent to a ${\rm SL} (2,\mathbb{R})$ matrix via
\begin{equation*}
Q = \frac{-1}{1+i} \begin{pmatrix} 1 & -i \\ 1 & i \end{pmatrix}
\end{equation*}
that is, $Q^{*}M^{z}_{\alpha_{n}}(x)Q \in {\rm SL} (2,\mathbb{R})$.
Since our Verblunsky coefficients $\alpha_{n}(x) \in \mathbb{D}$ are analytic and $\rho_{n}(x) = \sqrt{1-|\alpha_{n}(x)|^{2}}$, so $\|M_{n}(x)\| + \|M_{n}^{-1}(x)\| \leq C$ and our n-step transfer matrices satisfy condition (\ref{bounded}). 

Therefore we have LDT for Szeg\H o cocycle maps which is as follows,
\begin{lemma}
Assume $\omega$ satisfies a Diophantine condition (DC$_{A,c}$)
    \begin{equation*}
        \text{$\|k\omega\| > c |k|^{-A} $ for $k \in \mathbb{Z} \setminus \{ 0\}$}.
    \end{equation*}
Then there is $\sigma > 0$ such that
    \begin{equation*}
        \text{{\rm mes} $\{ x \in \mathbb{T} \vert \Big | \frac{1}{n} \log \| M^{z}_{n}(x) \| - L_{n}(z) \Big |  > n^{-\sigma} \} < e^{-n^{\sigma}}$}
    \end{equation*}
for all $n \in \mathbb{Z}_{+}$ and all $z \in \CZ$.
\end{lemma}
\begin{remark}
In \cite{BG00}, the statement of LDT is not uniform, but from the proof we can see it holds uniformly and to prove Anderson localization.
\end{remark}
In addition, it's easy to check the results in \cite[section 2 and section 3]{BG00} also hold true for $M_{n}^{z}(x)$.

\subsection{Elimination of the eigenvalue}
In \cite[section 4]{BG00}, the elimination of double resonances at a fixed point is proved based on the following fact for self-adjoint operators
\begin{equation*}
    \dist(E, \sigma(H_{n})) = \|(H_{n}-E)^{-1}\|^{-1}.
\end{equation*}
Although $\CC_{[0,n-1]}$ is not self-adjoint, even not normal, we can still get the same estimate as in \cite[Lemma 4.1]{BG00} using the following Lemma from \cite{DS06} instead of the above fact,
\begin{prop}
Let $\mathcal{M}_{n}$ be the set of pairs $(A,z)$, where $A$ is an $n \times n$ matrix, $z \in \mathbb{C}$ with
\begin{equation*}
|z| \geq \|A\|
\end{equation*}
and
\begin{equation*}
|z| \notin \sigma(A).
\end{equation*}
Then
\begin{equation*}
\sup_{\mathcal{M}_{n}} \dist(z,\sigma(A))\|(A-z)^{-1}\| = \cot(\frac{\pi}{4n}). 
\end{equation*}
\end{prop}

Then we have the CMV analog of \cite[Lemma 4.1]{BG00},
\begin{lemma}
Let $\log \log \bar{n}  \ll \log n$. Denote $S \subset \mathbb{T} \times \mathbb{T}$ the set of $(\omega, x)$ such that
    \begin{equation*}
        \text{$\|k\omega\| > c |k|^{-A} $ for $k \in \mathbb{Z} \setminus \{ 0\}$}.
    \end{equation*}

There is $n_{0} < \bar{n}$ and $z$ such that
    \begin{equation}{\label{resonance1}}
        \| (z-\CC_{[0,n_{0}-1]}(0))^{-1} \| > C^{n}
    \end{equation}
and
    \begin{equation}{\label{resonance2}}
        \frac{1}{n} \log \| M^{z}_{n}(x) \| < L_{n}(z)  -n^{\sigma}
    \end{equation}
Then
\begin{equation*}
    \text{{\rm mes} $S < e^{-\frac{1}{2}n^{\sigma}}$.}
\end{equation*}

\end{lemma}

\subsection{Semi-algebraic sets and frequency estimates}
According to the method in \cite{BG00}, in order to get Anderson localization, we need remove double resonances along the orbits $\{ x+n\omega \}$ where $n$ should be large enough. For i.i.d. distributed Verblunsky coefficients, this result is not so hard to obtain \cite{7authors17}, but for analytic Verblunsky coefficients, it becomes much more complicated. We need to rewrite the resonance conditions (\ref{resonance1}) and (\ref{resonance2}) as polynomials, then use the tool developed in \cite{M64} to estimate the complexity of $S$.

More specifically, we need to know the upper bound of interval numbers in the set $S_{x}=\{\omega \in \mathbb{T} | (\omega, x) \in S\}$ and we want to show it's at most polynomially large. By unitary conjugacy, all statements in \cite[section 5 and section 6]{BG00} hold true for the CMV case, which means we have the following estimate.

\begin{lemma}{\label{removedoubleresonances}}
Choose $\delta>0$ and $n$ a sufficiently large integer. Denote $\Omega_{n,\delta} \subset \mathbb{T}$ the set of frequencies $\omega \in DC_{10,c}$ such that

There is $n_{0}<n^{C}$, $2^{(\log n)^{2}} \leq \ell \leq 2^{(\log n)^{3}}$, and $z$ such that
\begin{equation*}
    \|(z-\CC_{[0,n_{0}-1]}(0))^{-1}\| > C^{n},
\end{equation*}
\begin{equation*}
    \frac{1}{n} \log \| M^{z}_{n}(\ell \omega) \| < L_{n}(z)  -n^{\sigma}
\end{equation*}
Then
\begin{equation*}
    \text{{\rm mes} $\Omega_{n,\delta}< 2^{-\frac{1}{4}(\log n)^{2}}$}.
\end{equation*}
\end{lemma}

\subsection{Proof of Theorem \ref{cmvandersonlocalization}} Now, we are ready to prove Anderson localization of CMV matrices with analytic Verblunsky coefficients using our formula (\ref{formulahalf}).

     Denote by $\Omega_{n,\delta}$ the frequency set obtained in Lemma 
\ref{removedoubleresonances} and define
    \begin{equation*}
        \Omega_{\delta} = \bigcap\limits_{n'} \bigcup\limits_{n>n'} \Omega_{n,\delta}, \quad \Omega = \bigcup\limits_{\delta} \Omega_{\delta},
    \end{equation*}
then  mes $\Omega = 0$.

Take $\omega \in \mathcal{I} \cap (DC_{10,c} \setminus \Omega)$ and let $z \in \CZ$, $\xi = (\xi_{n})_{n \in \mathbb{N}}$ satisfy the equation
\begin{equation*}
\CC(0) \xi = z \xi
\end{equation*}
where $\xi_{0}=1$ and $|\xi_{n}| \leq n^{C}$.

Let $\delta = \frac{\delta_{0}}{1000}$ and assume there is $n_{0}<n^{C}$ satisfies,
\begin{equation*}
    \|(z-\CC_{[0,n_{0}-1]}(0))^{-1}\| > C^{n},
\end{equation*}
then Lemma \ref{removedoubleresonances} tells us for all $2^{(\log n)^{2}} \leq \ell \leq 2^{(\log n)^{3}}$,  
\begin{equation}{\label{inequality1}}
    \frac{1}{n} \log \| M^{z}_{n}(\ell \omega) \| > L_{n}(z) - \delta.
\end{equation}
Similar to \cite[Lemma 2.1]{BG00}, we have
\begin{equation}{\label{inequality2}}
\frac{1}{n} \log \| M_{n}^{z}(x)\| < L_{n}(z) + \delta
\end{equation}
for all $z \in \CZ$.

According to Cramer's rule and direct calculations, we have
\begin{equation}{\label{cramerrule}}
|(z-\CC_{[0,n_{0}-1]})^{-1}(n_{1},n_{2})| \leq C \Big | \frac{\det(z-\tilde{\CC}_{[0,n_{1}-1]}) \det(z-\tilde{\CC}_{[n_{2}+1,n_{0}]})}{\det (z-\CC_{[0,n_{0}-1]})} \Big |
\end{equation}
where $|\det(z-\tilde{\CC}_{[a,b]})| \asymp |\det(z-\CC_{[a,b]})|$.

Next, Theorem \ref{halflinecase} allows us to estimate (\ref{cramerrule}).

\indent Recall that $\|M_{n}^{z}(x)\| = \|S_{n}^{z}(x)\|$, then Theorem \ref{halflinecase} implies
\begin{equation}{\label{inequality3}}
\begin{split}
 \|M_{n}^{z}(x)\| &\geq \prod\limits_{j=0}^{n-1}\frac{1}{\rho_{j}} (|\det(z-\CC_{[1,n-1]})| + |\det(z-\CC_{[0,n-1]})-z\det(z-\CC_{[1,n-1]})|)\\
                  &\geq \prod\limits_{j=0}^{n-1}\frac{1}{\rho_{j}} (|\det(z-\CC_{[1,n-1]})| + |\det(z-\CC_{[0,n-1]})|-|\det(z-\CC_{[1,n-1]})|)\\
                  &\geq \prod\limits_{j=0}^{n-1}\frac{1}{\rho_{j}} |\det(z-\CC_{[0,n-1]})|
\end{split}
\end{equation}
and 
\begin{equation*}
\text{$\|M_{n}^{z}(x)\| \leq 4 |\det(z-\CC_{[0,n-1]})|$ or $\|M_{n}^{z}(x)\| \leq 4 |\det(z-\CC_{[0,n-1]})-z\det(z-\CC_{[1,n-1]}))|$ }
\end{equation*}
which is equivalent to 
\begin{equation}{\label{inequality4}}
\|M_{n}^{z}(x)\| \leq 8 \max \{|\det(z-\CC_{[0,n-1]})|, |\det(z-\CC_{[1,n-1]})|\}.
\end{equation}

Apply inequalities (\ref{inequality1}), (\ref{inequality2}), (\ref{inequality3}) and (\ref{inequality4}) to (\ref{cramerrule}), then for each $2^{(\log n)^{2}} \leq \ell \leq 2^{(\log n)^{3}}$, one of the matrices
\begin{equation*}
    \CC_{[0,n-1]}(\ell \omega), \CC_{[1,n-1]}(\ell \omega)
\end{equation*}
thus $\CC_{\Lambda}(0)$ for some $\Lambda \in \{ [\ell, \ell+n-1], [\ell+1,\ell+n-1 ] \}$ will satisfy

\begin{equation*}
|G_{\Lambda}(n_{1},n_{2})|=|(z-\CC_{\Lambda})^{-1}(n_{1},n_{2})| \leq e^{-L_{n}(E)|n_{1}-n_{2}|+o(n)} \leq e^{-\delta_{0}L(E)+o(n)},
\end{equation*}
where we use the fact $\mathcal{C}_{[a,b]}(x+\omega)=\mathcal{C}^{T}_{[a+1,b+1]}(x)$ which means they have the same absolute value and the second inequality is because $L_{n}(z) < L(z) + \delta$ for n large.

Next, according to paving property in \cite{BG00}, for $2^{(\log n)^{2}+1}<N<2^{(\log n)^{3}-1}$, the Green's function $G_{[\frac{N}{2},N]}$ satisfies
\begin{equation*}
|G_{[\frac{N}{2},N]}(n_{1},n_{2})|<e^{-\delta |n_{1}-n_{2}|+o(N)}
\end{equation*}
Restricting the eigenequation $\CC(0)\xi=z\xi$ to the interval $[\frac{N}{2},N]$ and direct calculations imply
\begin{equation*}
|\xi_{N}| \leq e^{-\frac{\delta}{3}N}
\end{equation*}

It remains to show there is $n_{0}<n^{C}$ satisfies,
\begin{equation*}
    \|(z-\CC_{[0,n_{0}-1]}(0))^{-1}\| > C^{n},
\end{equation*}
which is almost the same with that in \cite{BG00}.

\section{Acknowledgements}
I would like to thank Zhenghe Zhang for suggesting me to figure out this formula and many useful discussions. I am grateful to my advisors David Damanik and Daxiong Piao for their partly supports.

\end{document}